\documentclass[reqno]{amsart}

\usepackage{a4wide}

\usepackage[hidelinks]{hyperref}

\usepackage[dvipsnames]{xcolor}
\definecolor{Mybgcolor}{HTML}{2E3440}
\definecolor{Mytextcolor}{HTML}{ECEFF4}
\definecolor{Mycommentcolor}{HTML}{88C0D0}
\definecolor{Myquestioncolor}{HTML}{B48EAD}
\definecolor{profcolor}{HTML}{EBCB8B}

\usepackage[shortlabels]{enumitem}
\setlist[enumerate,1]{label={\arabic*. }}

\usepackage{amsmath,amssymb}
\usepackage{microtype}
\usepackage{mathrsfs} 

\usepackage{tikz}

\newcommand{\dd}[1]{\mathrm{d}{#1}}

\newcommand{\loc}{\mathrm{loc}}

\newcommand{\dist}{\mathrm{dist}}

\newcommand\HH{\mathcal{H}}
\newcommand{\N}{\mathbb{N}}
\newcommand{\R}{\mathbb{R}}
\newcommand{\eps}{\varepsilon}
\newcommand{\mres}{\mathbin{\vrule height 1.6ex depth 0pt width
0.13ex\vrule height 0.13ex depth 0pt width 1.3ex}}

\newcommand{\fsln}{for sufficiently large $n\in \N$}

\newcounter{hints}
\renewcommand{\thehints}{\arabic{hints}}
\newcommand{\hintedrel}[2][]{
  \refstepcounter{hints}
  \if\relax\detokenize{#1}\relax\else\label{#1}\fi
  \mathrel{\overset{\mathrm{(\thehints)}}{\vphantom{\le}{#2}}}
}

\newtheorem{thm}{Theorem}[section]
\newtheorem{prop}[thm]{Proposition}
\newtheorem{lmm}[thm]{Lemma}

\theoremstyle{definition}

\theoremstyle{remark}

\numberwithin{equation}{section}

\begin{document}

\title[Boundary blow-ups for the Mumford--Shah functional]{Classification of boundary blow-ups\\for the Mumford--Shah functional\\with Dirichlet boundary conditions}
\author{Francesco Deangelis}
\address{Applied Mathematics M\"unster, University of M\"unster,
Einsteinstrasse 62, 48149 M\"unster, Germany}
\email{francesco.deangelis@uni-muenster.de}
\date{}

\begin{abstract}
We classify boundary blow-ups of minimizers of the Mumford--Shah functional
on bounded planar domains with Dirichlet boundary conditions. At every
point where the discontinuity set meets the boundary of the domain, we
prove that the discontinuity set of each blow-up is the tangent line to
the boundary at that point. This settles the conjecture formulated by
the author in 2024.
\end{abstract}

\maketitle

\section{Introduction}
\label{section: introduction}
The Mumford--Shah functional was introduced in \cite{MumfordShah} as
a variational model for image segmentation. Its two unknowns are a
function and a set across which the function may be discontinuous.
The competition between the Dirichlet energy of the function and the
length of this set gives rise to a free discontinuity problem: the
set is itself determined by minimization. In this paper we consider
the functional with prescribed Dirichlet boundary conditions and
study the geometry of the discontinuity set near the boundary.

Let $\Omega\subset\R^2$ be a bounded domain with boundary of class
$C^{1,\gamma}$, and let $g\in C^{1,\gamma}(\partial\Omega)$ for some
$\gamma\in(0,1)$. An admissible pair consists of a closed set
$K\subset\overline\Omega$ and a function
$u\in W^{1,2}(\Omega\setminus K)$ whose trace agrees with $g$ on
$\partial\Omega\setminus K$. We consider minimizers of
\begin{equation}
    \label{eq: mumford-shah functional}
    E(K,u) = \int_{\Omega \setminus K} |\nabla u|^2 + \HH^1(K)
\end{equation}
among all such pairs, where $\HH^1$ denotes one-dimensional Hausdorff
measure. The first term penalizes spatial variation of $u$ away from
$K$, while the second measures the length of $K$. In particular, the
length term includes $K\cap\partial\Omega$, whereas the boundary datum
is imposed only on $\partial\Omega\setminus K$. Thus the choice of
$K$ also determines the portion of the boundary on which $u$ must
attain the prescribed values. Existence of minimizers for this Dirichlet
problem was established by Carriero and Leaci \cite{carriero-leaci}.

This problem also arises in the variational approach to brittle fracture
under antiplane deformation \cite{francfort-marigo,bourdin-francfort-marigo}.
Here $u$ represents the displacement, $g$ the prescribed boundary displacement,
and $K$ the fracture set. The two terms in the functional account for
stored elastic energy and the energy required to create the fracture.

We call $K$ the discontinuity set. As in \cite{deangelis-MS}, we work
with normalized minimizers, meaning that
\[
 \HH^1(K\cap B_\rho(x))>0
 \qquad\text{for every }x\in K\text{ and }\rho>0,
\]
where $B_\rho(x)$ is the open Euclidean ball of radius $\rho$ centered
at $x$. This condition removes portions of the closed set that are
invisible to its length measure. No smoothness of $K$ is part of the
admissibility conditions. The regularity problem is to determine its
local geometry from minimality.

At interior points, the Mumford--Shah conjecture describes $K$ as a
locally finite union of $C^1$ arcs, with possible free endpoints and
triple junctions at angles of $120^\circ$. The author proved this
conjecture in \cite{deangelis-MS-conjecture}. A central tool in the
interior theory is to magnify a minimizer around a point and classify
the resulting limits, called blow-ups. Epsilon-regularity theorems
then turn sufficiently good approximation by one of these limiting
configurations into regularity at smaller scales.
For contributions to the regularity and classification theory, see
\cite{de-giorgi-carriero-leaci,dal-maso-morel-solimini,david-arcs,bonnet,AP97,AFP97,AFP99,Leg99,bonnet-david,david-leger,Lem11,de-lellis-focardi-density,de-lellis-focardi-integrability,bucur-luckhaus,de-philippis-figalli,Lem15,lemenant-mikayelyan,AM19,de-lellis-focardi-ghinassi,fischer,zhang,labourie-lemenant}.
A comprehensive account is given in \cite{delellis-focardi-book}.

At points of $K\cap\partial\Omega$, the prescribed boundary values enter the
limiting problem. The corresponding classification is the subject
of the present paper.

\subsection*{Boundary blow-ups}
Fix $x\in K\cap\partial\Omega$. For $r>0$, consider the rescaled
domain, set, and function
\[
 \Omega_{x,r}=\frac{\Omega-x}{r},\qquad
 K_{x,r}=\frac{K-x}{r},\qquad
 v_{x,r}(y)=\frac{u(x+ry)-g(x)}{\sqrt r}.
\]
The exponent $1/2$ is chosen so that, under this rescaling, both the
Dirichlet energy and the length are multiplied by $r^{-1}$. Let $T_x\partial\Omega$ denote
the tangent space to $\partial\Omega$ at $x$, and let $H_x$ be the
half-plane bounded by $T_x\partial\Omega$ on the interior side of $\Omega$. As $r\downarrow0$, the rescaled domains converge locally
to $H_x$. The subtraction of $g(x)$ and the regularity of $g$ ensure
that the rescaled boundary data tend locally to zero.

A boundary blow-up at $x$ is a limiting pair $(K_0,u_0)$ obtained
along a sequence $r_j\downarrow0$. The compactness theory of
\cite{deangelis-MS} gives local Hausdorff convergence of the sets
$K_{x,r_j}$ and convergence of the functions on complementary
components, after subtraction of componentwise additive constants
when necessary. The precise convergence is recalled in
Section~\ref{section: definitions and main results}. After a rotation,
we may take $H_x=U:=\{y_2>0\}$, so that $\partial U=\R\times\{0\}$.
The limiting set $K_0$ is a closed subset of $\overline U$ containing
the origin, and $u_0$ is locally Sobolev on $U\setminus K_0$, with
zero trace on $\partial U\setminus K_0$. On components joined to the lower
half-plane through $\partial U\setminus K_0$, this trace fixes the additive
normalization.

These limiting pairs are generalized minimizers of the half-plane
problem in the sense of \cite[Definition~2.5]{deangelis-MS}; the precise
definition is recalled in Section~\ref{section: definitions and main results}.

\subsection*{Previous results and the main theorem}
The boundary regularity theorem of \cite{deangelis-MS} shows that if
$x\in K\cap\partial\Omega$ and $K$ is sufficiently close to the tangent
line $x+T_x\partial\Omega$ to $\partial\Omega$ at $x$ in a small ball
centered at $x$, then $K$ is a $C^{1,\alpha}$ graph in a smaller ball,
for some $\alpha\in(0,1)$. This makes the
classification of boundary blow-ups the natural next step. The same
work conjectured that every such blow-up has discontinuity set
$T_x\partial\Omega$ and established partial classifications, including
the cases in which the limiting Dirichlet energy vanishes or the
limiting discontinuity set is connected.

Our classification result, Theorem~\ref{thm: classification at boundary contact},
settles this conjecture by classifying all normalized generalized
minimizers $(K,u)$ on $U$ whose discontinuity set meets $\partial U$. We prove
that, for some $c\in\R$,
\[
 K=\partial U,\qquad u\equiv c\quad\text{in }U.
\]
In particular, for every minimizer of
\eqref{eq: mumford-shah functional} and every
$x\in K\cap\partial\Omega$, each boundary blow-up has discontinuity
set $T_x\partial\Omega$. This identifies the limiting geometry at
every such point and supplies the classification required by the
boundary epsilon-regularity theory.

\subsection*{Global regularity for the Dirichlet problem}
Combining this classification with the boundary epsilon-regularity theorem
of \cite{deangelis-MS} and the solution of the Mumford--Shah conjecture
\cite{deangelis-MS-conjecture} gives a complete description of the
discontinuity set, including its behavior at the boundary. The conclusion
also holds for the full Mumford--Shah functional: for $f\in L^\infty(\Omega)$
and $\lambda\geq0$, set
\begin{equation}
\label{eq: full mumford-shah functional}
 E_\lambda(K,u)
 =\int_{\Omega\setminus K}\bigl(|\nabla u|^2+\lambda|u-f|^2\bigr)\,\dd{x}
   +\HH^1(K).
\end{equation}
Here $f$ is the fidelity datum, while $g$ remains the prescribed boundary datum.

\begin{thm}[Global regularity for the Dirichlet problem]
\label{thm: global Dirichlet regularity}
Let $\Omega\subset\R^2$ be a bounded domain with boundary of class
$C^{1,\gamma}$, let $g\in C^{1,\gamma}(\partial\Omega)$ for some
$\gamma\in(0,1)$, and let $f\in L^\infty(\Omega)$ and $\lambda\geq0$.
Let $(K,u)$ be a normalized minimizer of
\eqref{eq: full mumford-shah functional} among admissible pairs with
$u=g$ on $\partial\Omega\setminus K$.
Then $K$ is a finite union of regular $C^1$ arcs, with the following local structure:
\begin{enumerate}[label=(\roman*)]
\item At every point of $K\cap\Omega$, the set $K$ is locally a single
regular arc, a regular arc ending at that point, or three regular arcs
meeting only at that point at angles of $120^\circ$.
\item At every $x\in K\cap\partial\Omega$, the set $K$ is locally a single
$C^{1,\alpha}$ arc containing $x$ in its interior, for some $\alpha\in(0,1)$,
and its tangent line at $x$ is $x+T_x\partial\Omega$.
\end{enumerate}
In particular, the free endpoints and triple junctions are finite in number
and lie in $\Omega$.
\end{thm}
\begin{proof}
We first explain the inclusion of the fidelity term when $\lambda>0$.
By truncation and minimality, $|u|\leq M:=\max\{\|f\|_\infty,\|g\|_\infty\}$.
By truncation, it suffices to consider competitors whose functions take
values in $[-M,M]$. Since $f$ and these functions are bounded, the difference of the fidelity terms in a comparison
supported in $B_r(x)$ is $O(r^2)$. The rescaled energy is obtained by
dividing by $r$, so this contributes an $O(r)$ error, which vanishes as
$r\downarrow0$. The boundary compactness and epsilon-regularity proofs of
\cite{deangelis-MS} adapt to include these lower-order contributions.
In particular, boundary blow-ups are generalized minimizers of the same
homogeneous half-plane problem.

For $x\in K\cap\partial\Omega$, Theorem~\ref{thm: classification at boundary contact}
shows that every boundary blow-up has discontinuity set $T_x\partial\Omega$.
The boundary epsilon-regularity theorem
\cite[Theorem~1.1]{deangelis-MS}, with the extension just described,
gives a single $C^{1,\alpha}$ arc containing $x$ in its interior,
with tangent line $x+T_x\partial\Omega$ at $x$.

The interior description follows from
\cite[Theorem~1.1]{deangelis-MS-conjecture}. These local descriptions give
neighborhoods containing no free endpoints or triple junctions, except
possibly their centers in the interior case. Compactness of
$K\subset\overline\Omega$ then gives finitely many such points and a
finite cover by regular arcs, proving the global conclusion.
\end{proof}

For homogeneous Neumann boundary conditions, the analogous global
conclusion follows by combining \cite[Theorem~1.1]{deangelis-MS-conjecture}
with the boundary regularity result of David and L\'eger
\cite[Section~13]{david-leger}; see also \cite[Section~79]{david-book}.
For a bounded domain with $C^1$ boundary and a bounded fidelity datum,
the closure of the discontinuity set is a finite union of $C^1$ arcs
with the same interior structure, and the arcs that reach the boundary
end there orthogonally. Thus the interior and boundary results together
give the complete regularity description for both the Dirichlet and
homogeneous Neumann problems.

\subsection*{Proof strategy}
For a generalized minimizer $(K,u)$ on $U$, translate a point of
$K\cap \partial U$ to the origin and consider
\[
 K_R=\frac K R,\qquad u_R(y)=\frac{u(Ry)}{\sqrt R}.
\]
A limit of these pairs along a sequence $R\to\infty$, with the same
convergence and additive normalizations as above, is called a
blow-down. Boundary epsilon-regularity gives rigidity at infinity:
if a blow-down has discontinuity set $\partial U$, then $K=\partial U$
(Theorem~\ref{thm: rigidity at infinity}). A monotonicity argument
provides the circle-intersection criterion of
Proposition~\ref{prop: blow-ups s.t. K has 2 or more intersections with the circle of radius r for sufficiently large r}.
This criterion yields the classification when $\R^2\setminus K$ is
disconnected (Theorem~\ref{thm: blow-ups when K disconnects the plane}).

It remains to treat the case in which $\R^2\setminus K$ is connected.
Normal domain variations then define a nonnegative Radon measure
$\mu$ on $\partial U$, and a linear energy bound gives finite total mass
(Lemma~\ref{lmm: finite boundary reaction}). Away from $K$, this
measure has density equal to the square of the normal derivative of
$u$. Identifying $\partial U$ with $\R$, the measures associated with the
rescaled pairs satisfy $\mu_R(A)=\mu(RA)$ for Borel sets $A\subset\R$.
Their mass therefore tends to zero on every compact boundary interval
away from the origin as $R\to\infty$.

For a blow-down $(K_\infty,u_\infty)$, the density formula and harmonic
compactness yield zero Neumann data on boundary intervals disjoint from
$K_\infty$. Together with the inherited zero Dirichlet data, unique
continuation would give $u_\infty=0$ on $U\setminus K_\infty$ if
$\R^2\setminus K_\infty$ were connected. The classification of elementary
generalized minimizers \cite[Theorem~2.13]{deangelis-MS} would then give
$K_\infty=\partial U$, a contradiction. Hence $\R^2\setminus K_\infty$ is disconnected
for every blow-down (Lemma~\ref{lmm: separation in blowdown}). The classification
in that case gives a flat blow-down, and rigidity at infinity
concludes the proof.

Section~\ref{section: definitions and main results} fixes the
variational setting and states the main theorem.
Sections~\ref{section: rigidity proof}--\ref{section: disconnected proof}
establish rigidity at infinity, the circle-intersection criterion, and
the classification when $\R^2\setminus K$ is disconnected.
Section~\ref{section: classification at boundary contact} constructs
the measure on the boundary and completes the proof.

\section{Definitions and main result}
\label{section: definitions and main results}
\subsection*{The half-plane problem}
We set $U=\{x_2>0\}$ and recall the definition
of generalized minimizer from \cite[Definition~2.5]{deangelis-MS}.
An admissible pair consists of a closed set $K\subset\overline U$
and a function $u\in W^{1,2}_{\loc}(U\setminus K)$, with zero trace
on $\partial U\setminus K$, such that the local energy
\[
 E(K,u;O):=\int_{O\cap(U\setminus K)}|\nabla u|^2\,\dd{x}
                 +\HH^1(K\cap O)
\]
is finite for every bounded open set $O\subset\R^2$.

We call an admissible pair $(K,u)$ a \emph{generalized minimizer}
on $U$ if
\[
 E(K,u;O)\leq E(J,v;O)
\]
for every bounded open set $O\subset\R^2$ and every admissible pair
$(J,v)$ that agrees with $(K,u)$ outside $O$ and preserves the
separation of exterior points. More precisely, any two points in
$\R^2\setminus(O\cup K)$ belonging to distinct connected components
of $\R^2\setminus K$ must also belong to distinct connected components
of $\R^2\setminus J$.

As for the original minimizers, all generalized minimizers are normalized:
$\HH^1(K\cap B_r(z))>0$ for every $z\in K$ and $r>0$.
Moreover, $K$ is $\HH^1$-rectifiable, by the interior argument in
\cite[Appendix~D, Step~1]{delellis-focardi-book} and the inclusion
$K\cap \partial U\subset \partial U$.
We write $K_{x,r}=(K-x)/r$ and $u_{x,r}(y)=r^{-1/2}u(x+ry)$.

A generalized minimizer is called \emph{elementary} if
\[
 \int_{U\setminus K}|\nabla u|^2\,\dd{x}=0.
\]
This is the terminology of \cite[Definition~2.12]{deangelis-MS}.

\subsection*{Rescaling and limiting pairs}
For every $x\in \partial U$, the rescalings defined above preserve the
half-plane variational problem and satisfy
\[
 E(K_{x,r},u_{x,r};O)=\frac1r E(K,u;x+rO).
\]
A \emph{blow-up} is a subsequential limit along scales $r_j\downarrow0$;
a \emph{blow-down} is a subsequential limit along scales
$r_j\uparrow\infty$. Throughout the paper, convergence of pairs is
understood in the sense of
\cite[Proposition~2.3 and Theorem~2.9]{deangelis-MS}.
In particular, if $(K_j,u_j)\to(K_\infty,u_\infty)$, then
$K_j\to K_\infty$ locally in the Hausdorff distance, and the functions
converge in $C^\infty_{\loc}$ on each connected component of
$U\setminus K_\infty$, after subtracting suitable componentwise
additive constants when necessary. On components connected to the
lower half-plane through $\partial U\setminus K_\infty$, no constants need to
be subtracted, since the zero Dirichlet condition already fixes the
normalization. On these components, the convergence also holds locally
in $C^1$ up to $\partial U\setminus K_\infty$.

\subsection*{Main result}

The classification in Theorem~\ref{thm: classification at boundary contact}
below follows from Lemma~\ref{lmm: finite boundary reaction} and
Theorems~\ref{thm: blow-ups when K disconnects the plane} and
\ref{thm: rigidity at infinity}.

\begin{thm}[Classification of generalized minimizers meeting the boundary]
\label{thm: classification at boundary contact}
Let $(K,u)$ be a normalized generalized minimizer on $U=\{x_2>0\}$.
If $K\cap\partial U\neq\emptyset$, then
\[
 K=\partial U,\qquad u\equiv c\quad\hbox{in }U
\]
for some $c\in\R$.
\end{thm}

\section{Rigidity at infinity}
\label{section: rigidity proof}

First we remark that the epsilon-regularity theorem presented in \cite[Theorem 1.1]{deangelis-MS} is true also for generalized minimizers. Its proof is a simple adaptation of the proof given there. The corresponding statement is the following.
\begin{thm}[Epsilon-regularity for generalized minimizers]
    \label{thm: epsilon-regularity generalized minimizers}
    There exist constants $\eta,\varepsilon, \alpha \in (0,1)$ with the following property. Assume 
    \begin{enumerate}[label=(\roman*)]
        \item $(K,u)$ is a generalized minimizer on $U = \{x_2 > 0\}$ and $x \in K \cap \partial U$,
        \item $r^{-1} \dist_H(K \cap \overline{B}_{2r}(x), \partial U \cap \overline{B}_{2r}(x)) + r^\frac{1}{2} < \varepsilon$.
    \end{enumerate}
    Then $K \cap B_{\eta r}(x)$ is the graph of a $C^{1,\alpha}$ function.
\end{thm}

As a consequence of the epsilon-regularity theorem, we obtain the following rigidity result.
\begin{thm}[Rigidity at infinity]
    \label{thm: rigidity at infinity}
    Let $(K,u)$ be a generalized minimizer on $U = \{x_2 > 0\}$ with $0 \in K$. Suppose that $r_n \uparrow \infty$ and $K_{0, r_n} \to \partial U$ locally in the Hausdorff distance as $n \to \infty$. Then $K = \partial U$. 
\end{thm}
\begin{proof}
Let $\eta, \eps, \alpha$ be given by Theorem \ref{thm: epsilon-regularity generalized minimizers} and set $r = \tfrac{1}{4} \eps^2$. Define $(K_n, u_n) := (K_{0,r_n}, u_{0, r_n})$. Since $K_n \to \partial U$ locally in the Hausdorff distance, we have 
\begin{equation*}
    r^{-1} \dist_H(K_n \cap \overline{B}_{2r}, \partial U \cap \overline{B}_{2r}) + r^\frac{1}{2} < \varepsilon
\end{equation*}
\fsln. Moreover, $(K_n, u_n)$ is a generalized minimizer on $U$  and $0 \in K_n \cap \partial U$ for every $n \in \N$. Therefore, by Theorem \ref{thm: epsilon-regularity generalized minimizers}, $K_n \cap B_{\eta r}$ is the graph of a $C^{1,\alpha}$ function.
This implies that $K \cap B_{\eta r r_n}$ is the graph of a $C^{1,\alpha}$ function \fsln, and hence $K$ is connected. Therefore, by \cite[Proposition 2.15]{deangelis-MS}, $K = \partial U$.
\end{proof}

\section{A circle-intersection criterion}
\label{section: circle criterion}

\begin{prop}
    \label{prop: blow-ups s.t. K has 2 or more intersections with the circle of radius r for sufficiently large r}
    Let $(K,u)$ be a generalized minimizer on $U = \{x_2 > 0\}$ with $0 \in K$. Suppose that there exists $R>0$ such that $\HH^0(K \cap \partial B_r) \geq 2$ for all $r>R$. Then $K = \partial U$.
\end{prop}
\begin{proof}[Proof of Proposition \ref{prop: blow-ups s.t. K has 2 or more intersections with the circle of radius r for sufficiently large r}]
    Let $r_n \uparrow \infty$ as $n\to \infty$ and define $(K_n, u_n) := (K_{0,r_n}, u_{0,r_n})$ for all $n \in \N$. By \cite[Theorem 2.9]{deangelis-MS}, there is $(K_\infty, u_\infty)$ generalized minimizer on $U$ such that, up to subsequences, $(K_n, u_n)$ converges to $(K_\infty,u_\infty)$ in the sense of the theorem. 

    Then, by Lemma \ref{lmm: elementary blow-down if eventually 2 or more intersections for the circle}, $(K_\infty,u_\infty)$ is elementary. Therefore, since $0 \in K_\infty \cap \partial U$, $K_\infty = \partial U$ by \cite[Theorem 2.13]{deangelis-MS}.
    In particular, $K_{0,r_n} \to \partial U$ locally in the Hausdorff distance and Theorem \ref{thm: rigidity at infinity} implies that $K = \partial U$.
\end{proof}

We now prove the auxiliary lemma used above.

\begin{lmm}
    \label{lmm: elementary blow-down if eventually 2 or more intersections for the circle}
    Let $(K,u)$ satisfy the assumptions of Proposition \ref{prop: blow-ups s.t. K has 2 or more intersections with the circle of radius r for sufficiently large r}. Let $r_j \uparrow \infty$ and suppose that the rescalings $(K_{0,r_j}, u_{0,r_j})$ converge to $(K_\infty,u_\infty)$ in the sense of \cite[Theorem 2.9]{deangelis-MS}. Then $(K_\infty,u_\infty)$ is elementary. 
\end{lmm}
\begin{proof}
Define, for $r>0$,
\[
 F(r):=\frac{2}{r}\int_{B_r\cap U\setminus K}|\nabla u|^2
       +\frac{1}{r}\HH^1(K\cap B_r).
\]
By \cite[Proposition~2.10]{deangelis-MS}, $F$ is nondecreasing on
$(R,\infty)$. The standard energy upper bound
(cf.~\cite[Proposition~8.1]{deangelis-MS}) gives
$l_\infty:=\lim_{r\to\infty}F(r)<\infty$.
Moreover, the proof of \cite[Proposition~2.10]{deangelis-MS} yields
\begin{equation}
\label{eq: radial-energy-controlled-by-F}
 F'(\rho)\geq\frac{1}{\rho}
 \int_{\partial B_\rho\cap U\setminus K}
 |\partial_\nu u|^2\,\dd{\HH^1}
\end{equation}
for a.e. $\rho>R$.

Set $K_j=K_{0,r_j}$ and $u_j=u_{0,r_j}$, and fix $0<a<b$.
For $j$ large enough that $ar_j>R$, integrating
\eqref{eq: radial-energy-controlled-by-F} between $ar_j$ and $br_j$
(the singular part of the distributional derivative of $F$ is
nonnegative), then applying the coarea formula
\cite[Section~2.12, equation~(2.74)]{ambrosio-fusco-pallara} and
rescaling, gives
\begin{equation}
\label{eq: radial-energy-annulus}
 \int_{(B_b\setminus\overline B_a)\cap(U\setminus K_j)}
 \frac{1}{|x|}
 \left|\frac{x}{|x|}\cdot\nabla u_j(x)\right|^2\,\dd{x}
 \leq F(br_j)-F(ar_j)\longrightarrow0.
\end{equation}
Local uniform convergence of the gradients away from $K_\infty$
\cite[Proposition~2.3]{deangelis-MS}, together with Fatou's lemma
and the arbitrariness of $a,b$, therefore gives
\[
 \frac{x}{|x|}\cdot\nabla u_\infty(x)=0
 \qquad\text{for a.e. }x\in U\setminus K_\infty.
\]

Integrating by parts, using the weak Euler--Lagrange equation and the zero
Dirichlet condition on $\partial U\setminus K_\infty$, we obtain, for a.e.
$\rho>0$,
\[
 \int_{B_\rho\cap(U\setminus K_\infty)}
 |\nabla u_\infty|^2
 =\int_{\partial B_\rho\cap(U\setminus K_\infty)}
 u_\infty\,\partial_\nu u_\infty\,\dd{\HH^1}=0.
\]
Thus $(K_\infty,u_\infty)$ is elementary.
\end{proof}

\section{Disconnected complement}
\label{section: disconnected proof}

\begin{thm}
    \label{thm: blow-ups when K disconnects the plane}
    Let $(K,u)$ be a generalized minimizer on $U = \{x_2 > 0\}$ with $0 \in K$. Suppose that $\R^2 \setminus K$ is not connected. Then $K = \partial U$.
\end{thm}
\begin{proof}
Let $\Omega_*$ be the connected component of $\R^2 \setminus K$ containing the lower half-plane. Since $\R^2 \setminus K$ is not connected, there exists $C$ connected component of $\R^2 \setminus K$ different from $\Omega_*$. The argument in the proof of
\cite[Lemma~4.3.1]{delellis-focardi-book} also applies here and shows
that $C$ is unbounded.
In particular, since $C\neq\emptyset$, there is $R>0$ such that
$C\cap B_r\neq\emptyset$ for every $r>R$.
Since $C$ is connected and unbounded, we must have $C \cap \partial B_r \neq \emptyset$ for all $r > R$.
Observe that we cannot have $\HH^0(K \cap \partial B_r) \leq 1$ for some $r > R$, otherwise $\partial B_r$ connects $C$ and $\Omega_*$. Hence $\HH^0(K \cap \partial B_r) \geq 2$ for all $r > R$. We conclude by Proposition \ref{prop: blow-ups s.t. K has 2 or more intersections with the circle of radius r for sufficiently large r}.
\end{proof}

\section{Proof of the main theorem}
\label{section: classification at boundary contact}

We prove Theorem~\ref{thm: classification at boundary contact} using the
results above. The key step is to represent the first variation under
normal domain variations by a finite nonnegative measure on $\partial U$. Under
rescaling, its mass tends to zero on compact boundary intervals away from
the origin. Compactness and the classification when $\R^2\setminus K$
is disconnected then give a flat blow-down, to which rigidity at infinity
applies.

\begin{lmm}[A finite measure on the boundary]
\label{lmm: finite boundary reaction}
Let $(K,u)$ be a normalized generalized minimizer on $U$ such that
$\R^2\setminus K$ is connected.
There is a finite nonnegative Radon measure $\mu$ on $\R$ such that
\begin{equation}
\label{eq: flat reaction}
 \int_{U\setminus K}T_u:D(e_2\varphi)\,\dd{x}
 +\int_K(\tau\otimes\tau):D(e_2\varphi)\,\dd{\HH^1}
 =\int_{\R}\varphi(s,0)\,\dd{\mu}(s)
\end{equation}
for every $\varphi\in C_c^1(\R^2)$, where
$T_u=|\nabla u|^2I-2\nabla u\otimes\nabla u$ and $\tau$ is an
approximate unit tangent to $K$. On each open interval
$I\subset\R$ with $I\times\{0\}\subset \partial U\setminus K$,
\begin{equation}
\label{eq: flat reaction density}
 \mu\mathbin{\mres}I=|\partial_2u(s,0)|^2\,\dd{s}.
\end{equation}
For $r>0$, let $\mu_r$ be the measure associated with the rescaled pair
$(K_{0,r},u_{0,r})$ through \eqref{eq: flat reaction}. Then
\begin{equation}
\label{eq: flat reaction scaling}
 \mu_r(A)=\mu(rA)\qquad\text{for every Borel set }A\subset\R.
\end{equation}
\end{lmm}

\begin{proof}
The standard half-ball comparison
(cf.~\cite[Proposition~8.1]{deangelis-MS}) gives
\begin{equation}
\label{eq: halfplane linear energy}
 E(K,u;B_R)\leq(\pi+2)R\qquad(R>0).
\end{equation}

Extend $u$ by zero to the lower half-plane and denote the extension by
$\widehat u$. The zero trace gives
$\widehat u\in W^{1,2}_{\loc}(\R^2\setminus K)$. Since
$\R^2\setminus K$ is connected, the separation condition in the
definition of generalized minimizer is vacuous.

For $\varphi\in C_c^1(\R^2)$ with $\varphi(s,0)\geq0$, set
\[
 X=\varphi e_2,\qquad \Phi_t=\operatorname{Id}+tX.
\]
For $|t|\|DX\|_\infty<1$, the map $\Phi_t$ is a $C^1$
diffeomorphism of $\R^2$, and
$\Phi_t(\overline U)\subset\overline U$ for $t\geq0$. The pair
\[
 K_t=\Phi_t(K),\qquad
 w_t=(\widehat u\circ\Phi_t^{-1})|_U
\]
is admissible. The transported zero extension vanishes in the lower
half-plane.

Changing variables in the Dirichlet integral and applying the area formula
\cite[Theorem~2.91]{ambrosio-fusco-pallara} to the length term,
minimality shows that the left-hand side of
\eqref{eq: flat reaction}, denoted by $A(\varphi)$, is nonnegative.
If $\varphi|_{\partial U}=0$, applying this inequality to $\varphi$ and
$-\varphi$ gives $A(\varphi)=0$. By linearity, $A$ therefore depends only
on the restriction of $\varphi$ to $\partial U$. By the Riesz representation theorem
\cite[Corollary~1.55]{ambrosio-fusco-pallara}, the induced positive linear functional on $C_c^1(\R)$ is represented by
a locally finite nonnegative Radon measure $\mu$.

To prove finiteness on the whole line, take
$0\leq\chi\in C_c^1(B_2)$ equal to one on $B_1$ and set
$\chi_R(x)=\chi(x/R)$. Applying \eqref{eq: flat reaction}
with $\varphi=\chi_R$ and using \eqref{eq: halfplane linear energy}, we obtain
\begin{equation}
\label{eq: finite reaction mass}
 \mu([-R,R])
 \leq\frac{C}{R}
 \left(\int_{B_{2R}\cap(U\setminus K)}|\nabla u|^2\,\dd{x}
                   +\HH^1(K\cap B_{2R})\right)
 \leq C.
\end{equation}
Letting $R\uparrow\infty$ gives $\mu(\R)<\infty$.

On a boundary interval $I\times\{0\}\subset \partial U\setminus K$,
harmonic boundary regularity and $u=0$ give $\partial_1u=0$. Hence
\[
 (T_u(-e_2))\cdot e_2=|\partial_2u|^2
 \quad\hbox{on }I\times\{0\}.
\]
Using $\operatorname{div}T_u=0$, integration by parts in
\eqref{eq: flat reaction}, with test functions supported near
$I\times\{0\}$ and away from $K$, yields
\eqref{eq: flat reaction density}. Finally, inserting
$\varphi(\cdot/r)$ into \eqref{eq: flat reaction} and changing
variables in the bulk and surface terms proves
\eqref{eq: flat reaction scaling}; the critical factor $r^{-1/2}$
in $u_{0,r}$ leaves no additional factor in front of $\mu(rA)$.
\end{proof}

\begin{lmm}[Separation in a blow-down]
\label{lmm: separation in blowdown}
Let $(K,u)$ be a normalized generalized minimizer on $U$ with $0\in K$
such that $\R^2\setminus K$ is connected.
For every blow-down $(K_\infty,u_\infty)$ obtained from
\cite[Theorem~2.9]{deangelis-MS}, $\R^2\setminus K_\infty$ is disconnected.
\end{lmm}

\begin{proof}
Let $R_j\uparrow\infty$ give the blow-down, and write
$K_j=K_{0,R_j}$ and $u_j=u_{0,R_j}$. The limit is a normalized
generalized minimizer, $0\in K_\infty$, and the functions converge
without subtracting constants on components accessible from the lower
half-plane, as in \cite[Proposition~2.3 and Theorem~2.9]{deangelis-MS}.

Let $I\subset\R$ be a compact interval such that
$I\times\{0\}\subset \partial U\setminus K_\infty$ and set
$d=\operatorname{dist}(I,0)>0$. A neighborhood of $I\times\{0\}$
is disjoint from $K_j$ for large $j$. The measure of
Lemma~\ref{lmm: finite boundary reaction} satisfies
\begin{equation}
\label{eq: reaction tail}
 \int_I|\partial_2u_j(s,0)|^2\,\dd{s}
 =\mu(R_jI)
 \leq\mu(\{s:|s|\geq dR_j\})\longrightarrow0.
\end{equation}
Here the convergence follows from $\mu(\R)<\infty$.
By \cite[Proposition~2.3(i) and Theorem~2.9]{deangelis-MS},
$u_j\to u_\infty$ in $C^1$ up to $I\times\{0\}$.
Combining this with \eqref{eq: reaction tail} and the zero Dirichlet
condition gives
\[
 u_\infty=\partial_2u_\infty=0
 \quad\hbox{on }I\times\{0\}.
\]

Let $\Omega_*$ be the component of $\R^2\setminus K_\infty$ containing
the lower half-plane, and let $V$ be a component of
$U\setminus K_\infty$ contained in $\Omega_*$. A path in $\Omega_*$ from
$V$ to the lower half-plane first meets $\partial U$ at a point outside
$K_\infty$. A small ball there is disjoint from $K_\infty$ and its upper half
belongs to $V$. Thus $V$ reaches an interval of the preceding kind.
After shrinking the ball, both $u_\infty$ and $\partial_2u_\infty$
vanish on its diameter. The extension of $u_\infty$ by zero to the
lower half-ball is therefore weakly harmonic in the whole ball.
Since it vanishes on the lower half-ball, unique continuation gives
$u_\infty=0$ on the upper half-ball and then throughout $V$.

If $\R^2\setminus K_\infty$ were connected, every upper component
would be of this kind, so $u_\infty=0$ on $U\setminus K_\infty$.
Thus $(K_\infty,u_\infty)$ is elementary. Since $0\in K_\infty\cap \partial U$,
\cite[Theorem~2.13]{deangelis-MS} gives $K_\infty=\partial U$, contradicting
the connectedness of $\R^2\setminus K_\infty$.
\end{proof}

\begin{proof}[Proof of Theorem~\ref{thm: classification at boundary contact}]
Translate a point of $K\cap \partial U$ to $0$. To prove $K=\partial U$, suppose by
contradiction that $K\neq \partial U$. Then
Theorem~\ref{thm: blow-ups when K disconnects the plane} implies that
$\R^2\setminus K$ is connected. Extract a blow-down
$(K_\infty,u_\infty)$ using \cite[Theorem~2.9]{deangelis-MS}.
By Lemma~\ref{lmm: separation in blowdown},
$\R^2\setminus K_\infty$ is disconnected. Since $0\in K_\infty$,
Theorem~\ref{thm: blow-ups when K disconnects the plane} gives
$K_\infty=\partial U$. Therefore $K_{0,R_j}\to \partial U$ locally in the Hausdorff
distance for some $R_j\uparrow\infty$, and
Theorem~\ref{thm: rigidity at infinity} yields $K=\partial U$, a contradiction.
Since $K=\partial U$ is connected, \cite[Proposition~2.15]{deangelis-MS}
gives $\nabla u=0$ in $U$,
and hence $u$ is constant.
\end{proof}

\section*{Acknowledgements}
Funded by the Deutsche Forschungsgemeinschaft (DFG, German Research
Foundation) under Germany's Excellence Strategy EXC 2044/2 -- 390685587,
Mathematics M\"unster: Dynamics--Geometry--Structure.

The author used ChatGPT Astra to complete the proof of
Theorem~\ref{thm: classification at boundary contact},
building on his previous progress on the Dirichlet boundary problem. The author had also established further partial results,
omitted here because they follow as corollaries of this result.
The finiteness argument for the measure
in Lemma~\ref{lmm: finite boundary reaction}, the separation argument for
blow-downs in Lemma~\ref{lmm: separation in blowdown}, and the concluding
proof of Theorem~\ref{thm: classification at boundary contact} were
developed with the assistance of ChatGPT Astra. These are the arguments
in Section~\ref{section: classification at boundary contact}; they build
on the rigidity and topological results established earlier in the
manuscript and on the author's previous work \cite{deangelis-MS}.
ChatGPT was also used during the subsequent revision and simplification
of the mathematical arguments and their exposition.
The author has revised the manuscript and takes full responsibility for
its mathematical content and references.

\end{document}